\def \version {01-06-2021}
\documentclass[a4paper,12pt]{article}
\usepackage[utf8]{inputenc}
\usepackage{fullpage}
\usepackage{amssymb}
\usepackage{amsmath,enumitem}
\usepackage{tikz}
\usetikzlibrary{arrows}
\usepackage[final]{pdfpages}

\newtheorem{theorem}{Theorem}[section]
\newtheorem{lemma}[theorem]{Lemma}

\newtheorem{corollary}[theorem]{Corollary}
\newtheorem{fact}[theorem]{Fact}

\def \cay {\mbox{\rm Cay}}

\newenvironment{proof}[1][Proof]{\begin{trivlist}
\item[\hskip \labelsep {\bfseries #1\,}]}{\qed \end{trivlist}}

\newenvironment{remark}[1][Remark]{\begin{trivlist}
\item[\hskip \labelsep {\bfseries #1}]}{\end{trivlist}}

\newcommand{\sci}{{\chi_s'}}
\newcommand{\qed}{\hfill$\square$}
\title{Strong edge coloring of Cayley graphs and some product graphs}
\author{Suresh Dara$^1$, Suchismita Mishra$^2$, Narayanan Narayanan$^3$, Zsolt Tuza$^4$}
\date{
    \begin{small}
    $^1$Department of Mathematics, School of Advanced Sciences, VIT Bhopal University,\\ Kothri Kalan, Sehore-466114, India\\
    $2$ Advanced Computing and Microelectronics Unit, Indian Statistical Institute, Kolkata-700108, India\\
    $^3$Department of Mathematics, Indian Institute of Technology Madras, Chennai-600036, India.\\
    $^3$Alfr\'ed R\'enyi Institute of Mathematics, Budapest \& University of Pannonia, Veszpr\'em, Hungary.\\
    \end{small}
  \date{\small Latest update on \version}
}

\begin{document} 
\maketitle

\begin{abstract}
A strong edge coloring of a graph $G$ is a proper edge coloring of $G$ such that every color class is an induced matching. The minimum number of colors required is termed the strong chromatic index. In this paper we determine the exact value of the strong chromatic index of all unitary Cayley graphs. Our investigations reveal an underlying product structure from which  the unitary Cayley graphs emerge. We then go on to give tight bounds for the strong chromatic index of the Cartesian product of two trees, including an exact formula for the product in the case of  stars. Further, we give bounds for the strong chromatic index of the product of a tree with a cycle. For any tree, those bounds may differ from the actual value only by not more than a small additive constant (at most 2 for even cycles and at most 5 for odd cycles), moreover they yield the exact value when the length of the cycle is divisible by $4$.
\end{abstract}

\section{Introduction}
Throughout the paper, an edge joining vertices $u$ and $v$ is denoted by $(u,v)$. 
Let $G$ be a simple, finite, undirected graph. A {\it proper edge coloring} is a map $c$ from the edge set of $G$ to a set of distinct colors such that for any two edges $(u,v)$ and $(v,w)$, $c((u,v)) \neq c((v,w))$.
In addition, if  $c((u,v)) \neq c((w,x))$ whenever $(v,w)$ is an edge, then $c$ is called a {\it strong edge coloring}. 
 That is, a strong edge coloring is a proper edge coloring in which the vertex set of every color class induces a matching.
 The {\it strong chromatic index} of $G$, denoted by $\sci(G)$, is the minimum number of colors needed for any strong edge
coloring of $G$. 

There are several ways to look at this type of coloring.
Namely, the following conditions are equivalent:
 (a) each color class is an induced matching,
 (b) each $P_4\subset G$ is 3-edge-colored,
 (c) the square $(L(G))^2$ of the line graph of $G$ is properly vertex-colored,
 (d) the coloring is 1-intersection edge coloring\footnote{A $k$-intersection edge coloring,
   introduced in \cite{MutNarSub2009intersect}, is a proper edge coloring such that,
   for any two adjacent vertices, at most $k$ colors are incident with both.} of $G$.
The notion was introduced by Fouquet and Jolivet \cite{fouquet1983first} with
 the intention to represent conflict-free channel assignment in some radio networks.

There are many interesting conjectures on this problem.
Among them the most famous one is Erd\H os and Ne\v{s}et\v{r}il's conjecture,
 which states that the strong chromatic index of any graph $G$ is at most $(5\Delta^2 -2\Delta +1)/4$ if $\Delta$ is odd, and at most $5\Delta^2/4$ if $\Delta$ is even.
Here $\Delta=\Delta(G)$ denotes the maximum degree of the graph $G$.
As mentioned in \cite{faudree1989induced}, the conjecture was raised at
 the end of 1985; the first publication proposing the bound $5\Delta^2/4$
 seems to be \cite{erdos1988probl}.
If the bounds are valid, then they are tight, as shown by the graph obtained
 from the 5-cycle by substituting independent sets of size $\Delta/2$ into
 its vertices if $\Delta$ is even (inserting a complete bipartite graph between
 any two consecutive sets along the cycle), or two consecutive sets of size
  $(\Delta+1)/2$ and three others of size $(\Delta-1)/2$ if $\Delta$ is odd.

Since there are well over a hundred papers dealing with the strong chromatic index,
 here we necessarily are limited to a partial survey only.
The conjecture is proved by Chung et al.\ \cite{chung1990-2K2free} for all $\Delta$
 for the restricted class of graphs in which any two disjoint edges are joined
  by an edge (this particular case was conjectured already in 1983 by Bermond et al.\
 \cite{bermond1983bcc}, also raised independently in \cite{erdos1988probl}).

For small $\Delta$, Andersen \cite{andersen1992strong} and Hor\'ak et al.\
 \cite{horak1990cubic} showed that every cubic graph $G$ satisfies $\sci(G) \leq 10$.
(The case of non-regular sub-cubic graphs is easy, as noted in \cite{faudree1990sq}).
The case of $\Delta=4$ is already complicated, the upper bound of $20$ (respectively $19,18,17,16$) is known
 to be valid only under the further assumption that the maximum average degree
 is at most $51/13$ (respectively $15/4, 18/5, 7/2,61/18$) \cite{lv2018strong}.
These are improvements of the estimates in \cite{bensmail2015maxaverdeg}, where
 e.g.\ the sufficiency of maximum average degree at most $19/5$ for 20-colorability
 was proved.
The currently best result without average-degree restrictions states $\sci(G)\leq 21$, achieved only recently by Huang et al.\ \cite{huang2018strong}.
Since each edge is strongly independent from all but at most 24 other edges,
the upper bound $25$ is very easy.
Improving this bound to $24$, follows in one step from a general theorem of
 \cite{chung1990-2K2free} quoted above, applying Brooks's theorem.
This was further improved to $23$ by Hor\'ak \cite{horak1990maxdeg4}, to $22$ by Cranston \cite{cranston2006strong}, and finally to $21$ by Huang et al.\ \cite{huang2018strong}.

Bruhn and Joos \cite{bruhn2015stronger} proved that $\sci(G) \leq 1.93 \Delta^2(G)$, for graphs of sufficiently large maximum degree.
This improves an old bound of $\sci(G) \leq 1.998 \Delta^2(G)$ proved by Molloy and Reed \cite{molloy1997bound}. This bound is further improved for graphs with sufficiently large maximum degree. Hurley, Verclos and Kang showed that the strong chromatic index is atmost $1.772\Delta^2(G)$, for any graph G with sufficiently large maximum
degree $\Delta(G)$ \cite{hurley2021improved}. 

In 1989, Faudree et al.\ \cite{faudree1989induced} conjectured that every bipartite graph $G$ satisfies $\sci(G) \leq (\Delta(G))^2$.  
Brualdi and Quinn Massey strengthened this conjecture to state that if $G$ is a bipartite graph with bipartition $(A,B)$ and $\Delta(A)$ and $\Delta(B)$ are the maximum degrees
of the vertices in $A$ and $B$ respectively, then  $\sci(G) \leq
\Delta(A)\Delta(B)$ holds \cite{brualdi1993incidence}. 
Later, Nakprasit  \cite{nakprasit2008note} proved
that this conjecture is valid when $\Delta(A)=2$.  That is, for a $(2,\Delta)$-bipartite graph there is a strong edge
coloring that uses at most $2\Delta$ colors.  In 2017, Huang et al.\ \cite{huang2017strong} showed
that if $G$ is a $(3,\Delta)$-bipartite graph, then
 $\sci(G)\leq 3\Delta$. 
Bipartite graphs are complex also in the algorithmic sense:
 Mahdian \cite{mahdian2002computational} showed that determining the exact value of the
strong chromatic index is NP-hard even for bipartite graphs with girth at
least $g$, for any natural number $g$.

It is known that every planar graph admits a $4\Delta +4$ strong edge coloring \cite{faudree1990sq}. 
Moreover, $3\Delta+5$ colors are sufficient if the planar graph has girth $6$, and if it has girth at least $7$ then even $3\Delta$ colors are sufficient \cite{hudak2014strong}. Furthermore for outerplanar graphs  an exact formula can be given, as shown in the following theorem.

\vbox{
\begin{theorem}\cite{narayanan2015further}
Let $G$ be an outerplanar graph. Then $\sci(G)= \max \{\max \limits_{uv \in E}$ $d(u)+d(v)-1, \max \limits_{H \in \mathcal{P}} \sci(H)\}$, where $\mathcal{P}$ is the set of all puffer subgraphs of $G$. Moreover, if $G$ is bipartite, then $\sci(G)$ is either $\max_{uv \in E} d(u)+d(v)-1$ or $\max_{uv \in E} d(u)+d(v)$.
\end{theorem}
}

The exact values for the puffer graphs are obtained in \cite{chang2015strong}.

A Halin graph is a plane graph constructed from a tree $T$ without vertices of degree two by connecting all leaves through a cycle $C$. Let $G=T \cup C$ be a Halin graph. Lai et al.\ \cite{lai2012strong} proved that $\sci(G) \leq \sci(T)+3$, provided $G$ is different from some special graphs. 
Also it is known that every Halin graph $G$ with $\Delta(G) \geq 4$ satisfies $\sci(G) \leq 2\Delta(G)+1$ \cite{hu2018upper}.
Moreover, apart from two exceptions, cubic Halin graphs have $\chi'_s(G) \le 7$
 \cite{lih2012cubicHalin}.

We first discuss the strong edge coloring of the class of graphs called unitary Cayley graph. Let $U_n$ be the set of all units of $\mathbb{Z}_n$, that is $U_n = \{ m\in \mathbb{Z}_n\mid \gcd(m,n)=1\}$.
For any natural number the Cayley graph $\cay(\mathbb{Z}_n,U_n)$ is called a {\em unitary Cayley graph} and 
is denoted by $X_n$. Now for any $m \in \mathbb{Z}_n$, $\gcd(m,n) = 1$ if and only if $\gcd(-m,n)=1$, and
every element of $U_n$ generates $\mathbb{Z}_n$. Hence, $X_n$ is an undirected connected (and also Hamiltonian) graph.

The structure and various properties of unitary Cayley graphs
have been studied in the literature (see \cite{berrizbeitia}, \cite{klotz}). For a natural number $n$, $X_n$
is a $\phi(n)$-regular graph, where $\phi$ is Euler's phi function. 
It is known that a unitary Cayley graph is bipartite if and only if $n$ is even \cite{dejter}.
Akhtar et al.\ \cite{Akhtar} showed that the chromatic index 
of $X_n$ is $\phi(n)+1$ if $n$ is odd, and it is $\phi(n)$ otherwise. 

Here we  determine exactly the strong chromatic index of
Cayley graphs $X_n$, for all $n$. Namely, we prove that if a
given natural number $n$ has $k$ distinct prime factors
in its prime factor decomposition, then $\sci(X_n) = |E|/{2^{k-1}}$.

Besides the unitary Cayley graphs we also discuss about the Cartesian
product of graphs. Cartesian product of graphs is an important notion
in the theory of graph products, where the structure of the factors
(graphs) appears as an induced subgraph. 
For several important results and properties of the Cartesian product see the book \cite{imrich2008topics}.
Here we concentrate on the Cartesian product of certain classes of graphs.
The first result in the literature concerning strong chromatic index under the
 product operation was given in \cite{faudree1990sq} where the exact value for the
 $n$-dimensional hypercube was proven to be $\sci(Q_n)=2n$ for all $n\geq 2$.
The systematic study of $\sci$ on various types of graph products was initiated
 by Togni in \cite{togni2007strong}; in particular, he determined formulas for the
  Cartesian products of paths and cycles.
Complexity issues and general inequalities for some product types have also been
 presented by Chalermsook et al.\ in \cite{chalermsook2014product}.

Here we give both upper and lower bounds for the strong chromatic index of the
 Cartesian product of any two trees, and for the product of any tree with any cycle. 
If the length of a cycle $C$ is a multiple of $4$ then our results give the exact value
 of the strong chromatic index of the product of $C$ with an arbitrary tree.

\section*{Definitions and notation}

Following the standard notation, $C_{\ell}$ denotes the cycle graph of length $\ell$. We write $u \sim v$ to say $u$ is adjacent to $v$. 

For any edge $(u,v)$ in a graph $G$, define the edge degree $d'(u,v)$ of $(u,v)$ to be the number of edges incident 
either to $u$ or to $v$. 
That is, $d'(u,v)=d(u)+d(v)-1$.
The maximum edge degree of a graph, denoted by $\Delta'(G)$, is the maximum of all the edge degrees.

Recall that the \emph{Cartesian product} of any two graphs $G$ and $H$, denoted by $G \Box H$, is the graph with vertex set $\{ a:u\mid  a \in V(G)$ and $u \in V(H)\}$, where two vertices $a:u$ and $b:v$ are adjacent if either $a=b$ and $u$ is adjacent to $v$ in $H$,
or $u=v$ and $a$ is adjacent to $b$ in $G$.

For any vertex $h$ of $H$ define a $G$-{\it fiber} $G \times \{h\}$ to be the graph with vertex set $\{a:h\mid a \in V(G)\}$ and edge set $\{(a:h,b:h)\mid  (a,b) \in E(G)\}$. Similarly define $H$-fibers.
Note that we can write $G \Box H$ as $(G \times V(H)) \cup (V(G) \times H)$.

Another type of graph product is the \emph{categorical product} or \emph{direct product} of $G$ and $H$, denoted by $G\times H$.
It also has the vertex set $\{ a:u\mid  a \in V(G)$ and $u \in
V(H)\}$, but in this product vertices $a:u$ and $b:v$ are adjacent if
$a \sim b$ in $G$ and $u \sim v$ in $H$.
Notation is expressive for both kinds of products, indicating that $K_2\Box K_2\cong C_4$  and $K_2\times K_2\cong 2K_2$.

\section{Strong edge coloring of unitary Cayley graphs}

In this section we determine the strong chromatic index of all unitary Cayley graphs.

From its definition, it follows that the Cayley graph $X_n$ is a
$\phi(n)$-regular graph of order $n$. Therefore, the size (number of
edges) of $X_n$ is $n\frac{\phi(n)}{2}$.
Note that if $n$ has prime divisors $p_1,\dots,p_k$ then $\phi(n)=n\prod_{i=1}^k \left( 1 - \frac{1}{p_i}\right)$.
It will also be convenient to introduce the notation $n'=\frac{n}{\prod_{i=1}^k p_i}$,
 and to write $n''=\prod_{i=1}^k p_i$.

Our main result is the following formula.

\begin{theorem}\label{upperbound}
 Let\/ $n$ be a natural number with the prime factorization\/
 $p_1^{r_1}p_2^{r_2}\dots p_k^{r_k}$. Then\/
 $\sci(X_n) = \frac{ \vert E\vert}{2^{k-1}} = n\frac{\phi(n)}{2^{k}}$.
\end{theorem}

\begin{proof}
We  prove that $\frac{ \vert E\vert}{2^{k-1}}$ is both an upper bound and a lower bound on $\sci(X_n)$.
First let us make some observations on the structure of $X_n$.
Each $v\in\mathbb{Z}_n$ can be classified according to its residues modulo the prime factors, assigning with $v$ the $k$-tuple
 $$
   \langle v \rangle := (  \,v\!\!\!\!\mod p_1 \, , v\!\!\!\!\mod p_2 \, ,
    \, \dots \, , v\!\!\!\!\mod p_k \, ).
 $$
This assignment partitions the vertex set into $n''= p_1p_2...p_k$ classes, each class having cardinality $n' = p_1^{r_1-1}p_2^{r_2-1}...p_k^{r_k-1}$.
These are precisely the classes of false twins: if $\langle u \rangle = \langle v \rangle$, then $u$ and $v$ are not adjacent but they have exactly the same neighborhood.

As it can be seen directly from the definition, non-adjacency means that the two numbers in question are incongruent modulo each $p_i$.
In this way we can represent $X_n$ with a $k$-dimensional box, first taking the categorical product graph $K^*:=K_{p_1} \times K_{p_2} \times \cdots \times K_{p_k}$, and then substituting independent sets of size $n'$ (sets of false twins) into the product graph;
 that is, each edge of $K^*$ is enlarged to an induced copy of the complete bipartite graph $K_{n',n'}$. (This graph only rarely happens to be a product graph after
substitution.)

Let us note further that if $n=p^r$, i.e.\ in case $n$ has just one prime divisor,
 $X_n$ is the complete $p$-partite graph in which each vertex class has
 $p^{r-1}$ vertices.
Then no two edges of $X_n$ can form an induced matching, and
 the unique strong edge coloring assigns a distinct color to each edge.
For this reason,
  $$\sci(X_{p^r}) = \vert E(X_{p^r}) \vert = \frac{p^{2r-1}(p-1)}{2}$$
 clearly holds, and we only have to consider $k\geq 2$.

We first prove the upper bound given in the theorem.
\medskip

\noindent \underline{Proof of the upper bound:}

We show that the edge set of $X_n$ can be \emph{partitioned} into induced matchings of size $2^{k-1}$.
This clearly implies the inequality $\sci \leq \frac{ \vert E\vert}{2^{k-1}}$.

To this end, we  construct an edge partition for $K^*$, where  the
graph $K^*$ is actually just $X_{n''}$.
Consider any edge $(u,v)\in K_{p_1} \times \cdots \times K_{p_k}$.
Let us write $\langle u \rangle$ and $\langle v \rangle$ in the form
 $\langle u \rangle = (a_1,\dots,a_k)$ and $\langle v \rangle = (b_1,\dots,b_k)$.
Since $(u,v)$ is an edge, we have $a_i\neq b_i$ for all $1\le i\le k$.
Hence, the Cartesian product $\{a_1,b_1\} \times \cdots \times \{a_k,b_k\}$
 specifies a set, say $S$, of exactly $2^k$ vertices in $K^*$.
We claim that the subgraph induced by $S$ is a matching of size $2^{k-1}$.
Indeed, two vertices $(a'_1,\dots,a'_k),(b'_1,\dots,b'_k)\in S$ are adjacent
 if and only if $\{a'_i,b'_i\}=\{a_i,b_i\}$ holds for all $i$, as otherwise
 the corresponding two numbers would share a prime divisor.
Hence every $v\in S$ has one and only one neighbor in $S$; that is,
 $S$ induces a matching, which then necessarily has $|S|/2=2^{k-1}$ edges.
Note further that every edge induced by $S$ determines exactly the same set $S$,
 and so each edge of $K^*$ belongs to precisely one induced matching defined in this way.

Substituting independent sets of size $n'$ into the vertices of $K^*$, each edge
 gets replaced  by a subgraph isomorphic to $K_{n',n'}$. One such example is shown in Figure \ref{pic2}. Therefore each induced matching
 of size $2^{k-1}$ from $K^*$ becomes an induced subgraph isomorphic to
 $2^{k-1} K_{n',n'}$ in $X_n$.
Since the induced matchings defined in $K^*$ are mutually edge-disjoint,
 these copies of $2^{k-1} K_{n',n'}$ are edge-disjoint.
Obviously each of these subgraphs can be decomposed into $(n')^2$ edge-disjoint
 induced matchings of size $2^{k-1}$.
Consequently we obtain a required edge partition of $X_n$, and the upper bound
 $\sci \leq \frac{ \vert E\vert}{2^{k-1}}$ follows.

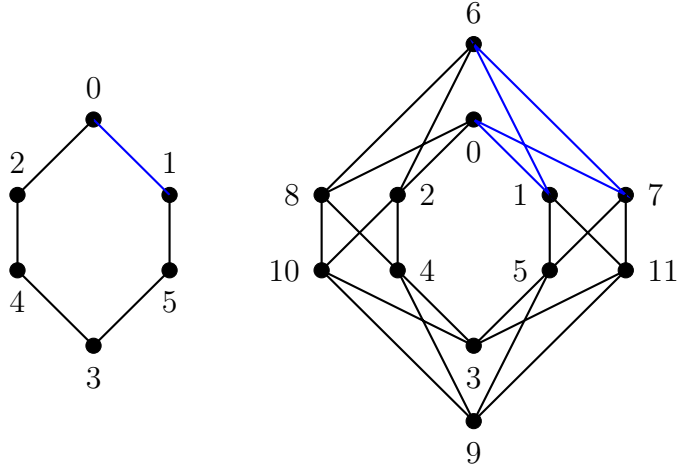
\begin{figure}[h]
\begin{center}
	
	\begin{tikzpicture}
		\draw[fill](0,0)circle[radius=0.1]node[black,above=4]{$0$};
		\draw[fill](1,-1)circle[radius=0.1]node[black,above=4]{$1$};
		\draw[fill](-1,-1)circle[radius=0.1]node[black,above=4]{$2$};
		\draw[fill](1,-2)circle[radius=0.1]node[black,below=4]{$5$};
		\draw[fill](-1,-2)circle[radius=0.1]node[black,below=4]{$4$};
		\draw[fill](0,-3)circle[radius=0.1]node[black,below=4]{$3$};
		
		\draw [blue, thick] (0,0)--(1,-1);
		\draw [black, thick] (1,-1)--(1,-2)--(0,-3)--(-1,-2)--(-1,-1)--(0,0);
		
		\draw[fill](5,0)circle[radius=0.1]node[black,below=4]{$0$};
		\draw[fill](6,-1)circle[radius=0.1]node[black,left=4]{$1$};
		\draw[fill](4,-1)circle[radius=0.1]node[black,right=4]{$2$};
		\draw[fill](6,-2)circle[radius=0.1]node[black,left=4]{$5$};
		\draw[fill](4,-2)circle[radius=0.1]node[black,right=4]{$4$};
		\draw[fill](5,-3)circle[radius=0.1]node[black,below=4]{$3$};
		
		\draw[fill](5,1)circle[radius=0.1]node[black,above=4]{$6$};
		\draw[fill](7,-1)circle[radius=0.1]node[black,right=4]{$7$};
		\draw[fill](3,-1)circle[radius=0.1]node[black,left=4]{$8$};
		\draw[fill](7,-2)circle[radius=0.1]node[black,right=4]{$11$};
		\draw[fill](3,-2)circle[radius=0.1]node[black,left=4]{$10$};
		\draw[fill](5,-4)circle[radius=0.1]node[black,below=4]{$9$};
		
		\draw [blue, thick] (5,0)--(6,-1)--(5,1)--(7,-1)--(5,0);
		\draw [black, thick] (6,-2)--(6,-1)--(7,-2)--(7,-1)--(6,-2);
		\draw [black, thick] (5,-3)--(6,-2)--(5,-4)--(7,-2)--(5,-3);
		\draw [black, thick] (5,-3)--(4,-2)--(5,-4)--(3,-2)--(5,-3);
		\draw [black, thick] (4,-1)--(4,-2)--(3,-1)--(3,-2)--(4,-1);
		\draw [black, thick] (4,-1)--(5,0)--(3,-1)--(5,1)--(4,-1);
	\end{tikzpicture}
	\caption{Substitution of an edge of $X_6$ in $X_{12}$} \label{pic2}	
\end{center}
\end{figure}	

\bigskip

In the proof of the lower bound we shall apply a particular case of the following
 theorem due to Alon \cite{Alon1985pairs}.
Originally the result was stated for pairs of \emph{$k$-tuples} of sets,
 we rewrite it with a somewhat simpler formalism dealing with pairs of sets.

\begin{lemma} \cite{Alon1985pairs}
  \label{setpairs}
Let\/ $Y_1,\dots,Y_k$ be mutually disjoint sets,\/
$Y=Y_1\cup\cdots\cup Y_k$. Moreover
 let\/ $s_1,\dots, s_k$ and\/ $t_1,\dots,t_k$ be positive integers.
Assume that\/
 $(A_1,B_1),\dots, (A_m,B_m)$ are pairs of sets with the following properties:
 \begin{itemize}
   \item $A_i\cup B_i\subseteq Y$ and\/ $A_i\cap B_i = \emptyset$ for all\/ $1\le i\le m$;
   \item $A_i\cap B_j \neq \emptyset$ for all\/ $1\le i<j\le m$;
   \item $|A_i\cap Y_\ell| \leq s_\ell$ and\/ $|B_i\cap Y_\ell| \leq t_\ell$ for all\/ $1\le i\le m$ and all\/ $1\le \ell\le k$.
 \end{itemize}
Then\/ $m\leq \prod_{\ell=1}^k {s_\ell+t_\ell \choose s_\ell}$.
\end{lemma}

Note that in case of $s_1=\ldots=s_k=t_1=\ldots=t_k=1$ the upper bound simply means $m\le 2^k$.

\bigskip

\noindent \underline{Proof of the lower bound:}

We  prove that the largest induced matchings in $X_n$ have no more than $2^{k-1}$ edges.  This clearly imply the lower bound $\sci \geq \frac{ \vert E\vert}{2^{k-1}}$.

Assume that the edges $(u_1,v_1),\dots,(u_h,v_h)$ form an induced matching.
We further  define $2h$ pairs of sets $(A_i,B_i)$ over an
underlying set $Y=Y_1\cup \cdots \cup Y_k$; two pairs of sets are to be defined for each edge.
For every $1\le\ell\le k$ we write $Y_\ell = \{y^\ell_0,y^\ell_1,\dots,y^\ell_{p_\ell-1}\}$;
 those vertices  represent the residue classes modulo $p_\ell$.

The ordered $k$-tuple $\langle v \rangle$ of integers allows us to associate a $k$-element subset $H(v)\subset Y$ with each vertex $v\in X_n$ as follows:
 $$
   H(v) = \{ y^1_{v\!\!\!\!\mod p_1} , y^2_{v\!\!\!\!\mod p_2} ,
     \dots, y^k_{v\!\!\!\!\mod p_k} \} \ .
 $$
Now, for every $1\le i\le h$ we set
 $$
   A_{2i-1} = H(u_i) , \quad B_{2i-1} = H(v_i) ,
    \quad A_{2i} = H(v_i) , \quad B_{2i} = H(u_i) \ .
 $$
Then the sets of the same index are disjoint, due to the adjacency of $u_i$ and $v_i$.
On the other hand, since $u_i$ is not adjacent to $u_j$ if $j\neq i$, the difference $u_i-u_j$ is divisible by some $p_\ell$, which implies that $A_{2i-1}$ --- as well as $B_{2i}$ --- meets both $B_{2j-1}$ and $A_{2j}$ inside $Y_\ell$.
Analogous consequences can be derived from the nonadjacencies
 $u_i\not\sim v_j$, $v_i\not\sim u_j$, and $v_i\not\sim v_j$.
It follows that the pairs $(A_1,B_1),\dots,(A_{2h},B_{2h})$ satisfy the conditions of Lemma \ref{setpairs}, with $s_\ell=t_\ell=1$ for all $1\le\ell\le k$.
Consequently $2h\le 2^k$ holds, so that every induced matching has at most $2^{k-1}$ edges, as claimed.
This completes the proof of the theorem.
\end{proof}

\section{Cartesian product of two trees}

In this section, we give both a lower and an upper bound for the strong chromatic index of the Cartesian product of two trees.
First we prove that twice the maximum degree is an upper bound for the strong chromatic index of the product of two trees. 
Recall the Brualdi and Quinn Massey's conjecture bound for the strong chromatic index of the bipartite graphs (the strong chromatic index of a $(A,B)$-partite graph is at most $\Delta(A)\Delta(B)$) \cite{brualdi1993incidence}. 
Note that the product of two trees is a bipartite graph. So our result implies the product of two trees satisfies the conjectured bound given by Brualdi and Quinn Massey.

\begin{theorem}\label{ub}
Let $T_1$ and $T_2$ be two trees. Then $ \sci(T_1 \Box T_2) \leq 2\Delta(T_1 \Box T_2)$. 
\end{theorem}

\begin{proof}
Let $x$ and $y$ be leaves of $T_1$ and $T_2$ respectively. We consider $T_1$ and $T_2$ as rooted trees with roots $x$ and $y$ respectively. Each of $x$ and $y$ has unique neighbors in their respective trees. Let $a$ be the only neighbor of $x$ in $T_1$ and $c$ be that of $y$ in $T_2$. 

We know that $2\Delta(T_1 \Box T_2) = 2\Delta(T_1) + 2\Delta(T_2)$. Now onwards we write $\Delta_1$ for $\Delta(T_1)$ and $\Delta_2$ for $\Delta(T_2)$. We give a coloring to the edges of $T_1 \Box T_2$.
The Cartesian product of the subgraph $S_1$, induced by $ \{x,a\}$ in $T_1$ and the subgraph  $S_2$, induced by $\{y,c\}$ in $T_2$ is just a $4$-cycle. Now give colors $2\Delta_1-1, 2\Delta_1, 2\Delta_1+2\Delta_2-1, 2\Delta_1+2\Delta_2$ to $(x:y,a:y), (a:c,x:c), (x:c,x:y), (a:y,a:c)$ respectively.

The idea is to construct a tower of induced subgraphs 
$G_1= S_1 \Box S_2 \subseteq G_2 \subseteq G_3 \subseteq ... \subseteq G_n = T_1 \Box T_2$ of $T_1 \Box T_2$ and in each step we extend the  coloring  of  $G_{i-1}$ to a strong edge coloring of $G_{i}$, so that the coloring satisfies the following conditions:
\begin{enumerate}[label=(\roman*)]
\item\label{ac} All the edges of $T_1 \times V(T_2)$ (recall it is collection of all the edges of the $T_1$-fibers) use colors from $\{1,2,3,...,2\Delta_1\}$.
\item\label{bd} All the edges of $V(T_1) \times T_2$ use colors from $\{2\Delta_1+1,2\Delta_1+2,...,2\Delta_1+
2\Delta_2\}$.
\item\label{b} The distance from $v$ to $v'$ is even in $T_2$ if and only if for every $(u,u')\in E(T_1)$, the color of the edge $(u:v,u':v)$ is same as the color of $(u:v',u':v')$.
\item\label{a} The distance from $u$ to $u'$ is even in $T_1$ if and only if for every $(v,v') \in E(T_2)$, the color of the edge $(u:v,u:v')$ is same as the color of $(u':v,u':v')$.
\end{enumerate} 

Loosely speaking this will be carried out by visiting carefully selected
 representative vertices in a sequence where the incident uncolored edges are
  in an interaction with relatively few previously colored edges; and then
  copy the colors to further subsets of edges in the current subgraph.
 The procedure will be performed on $T_1$ and $T_2$ separately.
This approach can be formalised in a precise way as described below.

We define $G_2$ to be the subgraph induced by $\{u:v\mid  d_{T_1}(x,u) \leq 2, \ d_{T_2}(y,v) \leq 2\}$ of $T_1 \Box T_2$. 
Now $|\{(a:y,a':y),(a:c,a':c) \mid  a'$ is adjacent to $a$ and $a' \neq x\}| \leq 2\Delta_1-2$. So we can assign colors $1,2,...,2\Delta_1-2$ to these edges such that no two edges get the same color. Similarly we can distribute the colors $2\Delta_1+1, 2\Delta_1+2,...,2\Delta_1+2\Delta_2-2$ to $\{(x:c,x:c'),(a:c,a:c')\mid  c'$ is adjacent to $c$ and $c' \neq y\}$, so that no two edges get the same color.

\begin{figure}[h]
\begin{center}
	\definecolor{qqqqff}{rgb}{0,0,1}
	\begin{center}
		\begin{tikzpicture}[thick, scale=0.83][line cap=round,line join=round,>=triangle 45,x=1.0cm,y=1.0cm]
			\clip(3,-12.88) rectangle (24.59,0.39);
			\draw (5.52,-6.09)-- (5.52,-4.09)
			node[sloped,pos=0.5,above]{\begin{tiny}
					$2\Delta_1-1$
			\end{tiny}};
			\draw (5.52,-6.09)-- (6.48,-7.43)
			node[sloped,pos=0.5,above]{\begin{tiny}
					$1$
			\end{tiny}};
			\draw (6.48,-7.43)-- (6.52,-9.09);
			\draw (5.52,-10.09)-- (4.52,-8.09);
			\draw (4.52,-8.09)-- (3.52,-10.09);
			\draw (4.52,-8.09)-- (5.52,-6.09)
			node[sloped,pos=0.5,above]{\begin{tiny}
					$2$
			\end{tiny}};
			\draw (10.72,-3.5)-- (10.72,-1.5)
			node[sloped,pos=0.5,above]{\begin{tiny}
					$2\Delta_1$
			\end{tiny}};
			\draw (10.72,-3.5)-- (11.68,-4.84)
			node[sloped,pos=0.5,below]{\begin{tiny}
					$\Delta_1+1$
			\end{tiny}};
			\draw (11.68,-4.84)-- (11.72,-6.5);
			\draw (10.7,-7.7)-- (9.72,-5.5);
			\draw (9.72,-5.5)-- (8.72,-7.5);
			\draw (9.72,-5.5)-- (10.72,-3.5)
			node[sloped,pos=0.5,above]{\begin{tiny}
					$\Delta_1+2$
			\end{tiny}};
			\draw (16.61,-7.51)-- (16.61,-5.51);
			\draw (16.61,-7.51)-- (17.57,-8.85);
			\draw (17.57,-8.85)-- (17.61,-10.51);
			\draw (16.61,-11.51)-- (15.61,-9.51);
			\draw (15.61,-9.51)-- (14.61,-11.51);
			\draw (15.61,-9.51)-- (16.61,-7.51);
			\draw (10.72,-1.5)-- (5.52,-4.09)
			node[sloped,pos=0.5,above]{\begin{tiny}
					$2\Delta_1+2\Delta_2-1$
			\end{tiny}};
			\draw (19.92,-2.64)-- (19.92,-0.64);
			\draw (19.92,-2.64)-- (20.88,-3.98);
			\draw (20.88,-3.98)-- (20.92,-5.64);
			\draw (20.02,-6.84)-- (18.92,-4.64);
			\draw (18.92,-4.64)-- (17.92,-6.64);
			\draw (18.92,-4.64)-- (19.92,-2.64);
			\draw (10.72,-1.5)-- (19.92,-0.64)
			node[sloped,pos=0.5,above]{\begin{tiny}
					$2\Delta_1+11$
			\end{tiny}};
			\draw (10.72,-1.5)-- (16.61,-5.51)
			node[sloped,pos=0.6,above]{\begin{tiny}
					$2\Delta_1+2$
			\end{tiny}};
			\draw (10.72,-3.5)-- (5.52,-6.09)
			node[sloped,pos=0.5,above]{\begin{tiny}
					$2\Delta_1+2\Delta_2$
			\end{tiny}};
			\draw (10.72,-3.5)-- (19.92,-2.64)
			node[sloped,pos=0.5,above]{\begin{tiny}
					$2\Delta_1++\Delta_2+1$
			\end{tiny}};
			\draw (10.72,-3.5)-- (16.61,-7.51)
			node[sloped,pos=0.5,above]{\begin{tiny}
					$2\Delta_1+\Delta_2+2$
			\end{tiny}};
			\draw (11.68,-4.84)-- (6.48,-7.43);
			\draw (11.68,-4.84)-- (17.57,-8.85);
			\draw (11.68,-4.84)-- (20.88,-3.98);
			\draw (6.52,-9.09)-- (11.72,-6.5);
			\draw (11.72,-6.5)-- (17.61,-10.51);
			\draw (11.72,-6.5)-- (20.92,-5.64);
			\draw (10.7,-7.7)-- (16.61,-11.51);
			\draw (10.7,-7.7)-- (20.02,-6.84);
			\draw (10.7,-7.7)-- (5.52,-10.09);
			\draw (8.72,-7.5)-- (3.52,-10.09);
			\draw (8.72,-7.5)-- (14.61,-11.51);
			\draw (8.72,-7.5)-- (17.92,-6.64);
			\begin{scriptsize}
				\fill [color=qqqqff] (5.52,-4.09) circle (1.5pt);
				\draw[color=qqqqff] (5.5,-3.86) node {$x:y$};
				\fill [color=qqqqff] (5.52,-6.09) circle (1.5pt);
				\draw[color=qqqqff] (5,-6) node {$x:c$};
				\fill [color=qqqqff] (6.48,-7.43) circle (1.5pt);
				\draw[color=qqqqff] (6,-7.5) node {$x:c_1$};
				\fill [color=qqqqff] (4.52,-8.09) circle (1.5pt);
				\draw[color=qqqqff] (4,-7.86) node {$x:c_2$};
				\fill [color=qqqqff] (6.52,-9.09) circle (1.5pt);
				\draw[color=qqqqff] (6.66,-8.86) node {};
				\fill [color=qqqqff] (5.52,-10.09) circle (1.5pt);
				\draw[color=qqqqff] (5.64,-9.86) node {};
				\fill [color=qqqqff] (3.52,-10.09) circle (1.5pt);
				\draw[color=qqqqff] (3.66,-9.86) node {};
				\fill [color=qqqqff] (10.72,-1.5) circle (1.5pt);
				\draw[color=qqqqff] (10.99,-1.26) node {$a:y$};
				\fill [color=qqqqff] (10.72,-3.5) circle (1.5pt);
				\draw[color=qqqqff] (10,-3.4) node {$a:c$};
				\fill [color=qqqqff] (11.68,-4.84) circle (1.5pt);
				\draw[color=qqqqff] (11,-5.5) node {$a:c_1$};
				\fill [color=qqqqff] (9.72,-5.5) circle (1.5pt);
				\draw[color=qqqqff] (9,-5.3) node {$a:c_2$};
				\fill [color=qqqqff] (11.72,-6.5) circle (1.5pt);
				\draw[color=qqqqff] (11.99,-6.26) node {};
				\fill [color=qqqqff] (10.7,-7.7) circle (1.5pt);
				\draw[color=qqqqff] (10.95,-7.46) node {};
				\fill [color=qqqqff] (8.72,-7.5) circle (1.5pt);
				\draw[color=qqqqff] (8.99,-7.26) node {};
				\fill [color=qqqqff] (16.61,-5.51) circle (1.5pt);
				\draw[color=qqqqff] (16.89,-5.26) node {$a_1:y$};
				\fill [color=qqqqff] (16.61,-7.51) circle (1.5pt);
				\draw[color=qqqqff] (17.3,-7.4) node {$a_1:c$};
				\fill [color=qqqqff] (17.57,-8.85) circle (1.5pt);
				\draw[color=qqqqff] (18,-8.61) node {$a_1:c_1$};
				\fill [color=qqqqff] (15.61,-9.51) circle (1.5pt);
				\draw[color=qqqqff] (15,-9.28) node {$a_1:c_2$};
				\fill [color=qqqqff] (17.61,-10.51) circle (1.5pt);
				\draw[color=qqqqff] (17.89,-10.28) node {};
				\fill [color=qqqqff] (16.61,-11.51) circle (1.5pt);
				\draw[color=qqqqff] (16.88,-11.28) node {};
				\fill [color=qqqqff] (14.61,-11.51) circle (1.5pt);
				\draw[color=qqqqff] (14.89,-11.28) node {};
				\fill [color=qqqqff] (19.92,-0.64) circle (1.5pt);
				\draw[color=qqqqff] (20.2,-0.41) node {$a_2:y$};
				\fill [color=qqqqff] (19.92,-2.64) circle (1.5pt);
				\draw[color=qqqqff] (20.5,-2.41) node {$a_2:c$};
				\fill [color=qqqqff] (20.88,-3.98) circle (1.5pt);
				\draw[color=qqqqff] (21.5,-3.73) node {$a_2:c_1$};
				\fill [color=qqqqff] (18.92,-4.64) circle (1.5pt);
				\draw[color=qqqqff] (20,-4.41) node {$a_2:c_2$};
				\fill [color=qqqqff] (20.92,-5.64) circle (1.5pt);
				\draw[color=qqqqff] (21.2,-5.41) node {};
				\fill [color=qqqqff] (20.02,-6.84) circle (1.5pt);
				\draw[color=qqqqff] (20.28,-6.61) node {};
				\fill [color=qqqqff] (17.92,-6.64) circle (1.5pt);
				\draw[color=qqqqff] (18.2,-6.41) node {};
			\end{scriptsize}
		\end{tikzpicture}
	\end{center}
\end{center}
\end{figure}

Since $T_1$ is a tree, if $a'$ and $a''$ are adjacent to $a$, then they are not adjacent to each other,
 therefore $\{ (a':y,a':c)\mid  a' \sim a\}$ forms an induced matching; and so does
  $\{ (a':c,a':c')\mid  a' \sim a\}$ for any fixed neighbor $c'$ of $c$ in $T_2$.
Hence we can assign the color of $(x:y,x:c)$  to $\{ (a':y,a':c)\mid  a'\neq x, \ a' \sim a\}$ and the color of  $(x:c,x:c')$ to $\{ (a':c,a':c')\mid  a' \neq x, \ a' \sim a\}$. 
Similarly, for every fixed $a' \sim a$, we can give the color of $(a':y,a:y)$ to $\{ (a:c',a':c')\mid c'\neq y, \ c' \sim c\}$. Clearly this coloring satisfies \ref{ac}, \ref{bd}, \ref{b}, \ref{a}. 

Now for every $\ell$, define $G_{\ell}$ to be the subgraph induced by $\{u:v\mid  d_{T_1}(x,u) \leq \ell, \ d_{T_2}(y,v) \leq \ell\}$ (note that if $m = \max\{ \mathrm{diam}(T_1), \mathrm{diam}(T_2)\}$, then $G_m = T_1 \Box T_2$).
Suppose we have a coloring of $G_{\ell-1}$ that satisfies \ref{ac}, \ref{bd}, \ref{b}, \ref{a}. Now we give coloring to the rest of the edges of $G_{\ell}$. First we assign the colors to $\{ (u:v,u':v)\mid  (u,u') \in E(T_1),v \in V(T_2)\}$.

 Let $u$ be a vertex of $T_1$ of distance $\ell-1$ from $x$ and $s$ be its parent. 
The colored edges of distance two from $\{ (u:y,u':y)\mid  u'\neq s, \ u' \sim u\}$ are the edges incident to $s:y$ or $u:c$. Since $G_{\ell-1}$ satisfies \ref{ac} and \ref{bd}, so $(s:c,u:c)$ is the only edge adjacent to $u:c$ that uses a color among $1,2,3,...,2\Delta_1$.
Again, coloring of $G_{i-1}$ satisfies \ref{b}, there are at least $\Delta_1-1$ colors among $\{1,2,3,...,2\Delta_1\}$, not being used by any of the edges incident to $s:y$ or by the edge $(s:c,u:c)$. We can assign those colors to $\{ (u:y,u':y)\mid  u' \neq s, \ u' \sim u\}$, so that no two edges get the same color.
Then for any fixed $u'\neq s$ where $u' \sim u$, the only colored edges of distance two from $(u:c,u':c)$ are the edges incident to $s:c$ or incident to one of the vertices of the set $\{u:c'\mid c' \sim c\}$. 
Now for any $c'\neq y$ adjacent to $c$, $(s:c',u:c')$ is the only edge incident to  $u:c'$ that uses colors among $1,2,3,...,2\Delta_1$ (by \ref{ac}, \ref{bd}). Again \ref{b} says that the color of  $(s:c',u:c')$ is the color of $(s:y,u:y)$.
Therefore, there are $\Delta_1-1$ suitable colors among $1,2,3,...,2\Delta_1$ for $\{ (u:c,u':c)\mid  u'\neq s, \ u' \sim u\}$. Distribute those colors to $\{ (u:c,u':c)\mid  u'\neq s, \ u' \sim u\}$.

Now we claim the following for a fixed $u' \neq s$, where $u'$ is adjacent to $u$.

\medskip

\textit{Claim:} We can assign the color of $(u:v,u':v)$ to $\{ (u:t,u':t)\mid  d_{T_2}(v,t) = 2\}$.

\textit{Proof of the Claim:} We apply induction on distance from $u$. If $(u:t,u':t)$ is already colored, then the induction hypothesis assures the claim. Now we may assume that $(u:t,u':t)$ did not get a color till now and $v'$ is adjacent to both $v$ and $t$. The only colored edges of distance two from $(u:t,u':t)$ are $\{ (s:t',u:t')\mid  t'$ is a child of $t\}$ and the edges incident to $u:v'$ or $s:t$. Since the coloring of $G_{\ell-1}$ satisfies \ref{b}, the only colors which may create problems are the colors of the edges incident to $u:v'$ or $s:t$.

Since $v'$ is adjacent to $t$, the color of $(u:v,u':v)$ is not used by any of the edges incident to $u:v'$. Also \ref{b} ensures that it is not used by the edges incident to $s:t$. Again $\{ (u:t,u':t)\mid  d_{T_2}(v,t) = 2\}$ is an induced matching. Hence we can assign the color of $(u:v,u':v)$ to $\{ (u:t,u':t)\mid  d_{T_2}(v,t) = 2\}$.
The claim is proved. By coloring accordingly we get a partial strong edge coloring that satisfies \ref{ac} and \ref{b}.

\medskip

By symmetry we can color $V(T_1) \times T_2$ by using the colors $2\Delta_1+1, 2\Delta_1+2,...,2\Delta_1+2\Delta_2$ such that it satisfies \ref{bd} and \ref{a}.
That says the strong chromatic index of $G_n = T_1 \Box T_2$ is at most $2\Delta_1 + 2\Delta_2 = 2\Delta(T_1 \Box T_2)$.
\end{proof}

In the above theorem we observed an upper bound for the Cartesian product of two trees. This bound is tight. That is there exist trees $T_1$ and $T_2$ such that $\sci(T_1 \Box T_2) = 2\Delta(T_1 \Box T_2)$. Before showing that, we find a lower bound for the same, in the next theorem.

\begin{theorem}\label{lowerBound}
Let $T_1$ and $T_2$ be two trees. Then
 $$\sci(T_1 \Box T_2) \geq \max \{2\Delta(T_1)+\Delta'(T_2), 2\Delta(T_2)+\Delta'(T_1)\}+1 = \Delta'(T_1 \Box T_2)+1 .$$
\end{theorem}

\begin{proof}
Let $(a,b)$ be an edge with maximum edge degree in $T_1$. Also let $c$ be a vertex of maximum degree in $T_2$ and $d$ be a neighbor of it. %Clearly $S \Box \{c\}$ is an induced subgraph of $T_1 \Box T_2$,
Now the edge $(a:d,b:d)$ and all the edges incident to $a:c$ or $b:c$ should get different colors. So $\sci(T_1 \Box T_2) \geq \Delta'(T_1) + 2\Delta(T_2)+1$. Similarly, $\sci(T_1 \Box T_2) \geq \Delta'(T_2) + 2\Delta(T_1)+1$. Therefore, 
$\sci(T_1 \Box T_2) \geq \max \{ \Delta'(T_1) + 2\Delta(T_2), \Delta'(T_2) + 2\Delta(T_1)\}+1 = \Delta'(T_1 \Box T_2) +1 $.
\end{proof}

Let $T_1$ and $T_2$ be two trees and suppose that one of them has two maximum-degree vertices adjacent to each other. Then the upper and lower bound of $\sci(T_1 \Box T_2)$ are the same. So the following corollary is immediate.

\begin{corollary}   \label{cor:path}
Let $T$ be a tree. Then $\chi _s'(T \Box P_n) = 2(\Delta (T)+\Delta (P_n)) = 2\Delta (T \Box P_n)$ for every $n>3$, or $n=2$.
\end{corollary}

In general the upper bound does not hold with equality. An example is the Cartesian product of two stars. Its strong chromatic index is determined in the following theorem.

\begin{theorem}
% Let $S_1$, $S_2$ be two star graphs of size $n$ and $m$  respectively. Then, 
The strong chromatic index of $K_{1,n} \Box K_{1,m}$ is  $2n+m+2$, where $n\geq m$.
\end{theorem}

\begin{proof}
 Let $(\{x\},\{x_1,x_2,...,x_n\})$ and $(\{y\},\{y_1,...,y_m\})$ be the bipartition of  $K_{1,n}$ and $K_{1,m}$ respectively.
We first give a strong edge coloring on $K_{1,n} \Box K_{1,m}$ with $2n+m+2$ colors as follows.

\begin{itemize}
  \item $n+m$ colors on the edges incident with the vertex $x:y$, one dedicated
    color for each edge.
  \item $n-m$ colors on the `long' edges of the $K_{1,n}$-fibers, namely one
    color for each induced matching of type $\{(x:y_j,x_i:y_j)\mid 1\leq j\leq m\}$,
    in the range $m+1\leq i\leq n$.
  \item $m$ colors on induced matchings which take one edge from each of $n-1$
    $K_{1,m}$-fibers and one from each of $m-1$ $K_{1,n}$-fibers, defined as
    $$ A_i := \{ (x:y_i, x_j:y_i) \mid 1\leq j\leq n, \ j\neq i \}
     \cup \{ (x_i:y, x_i:y_j) \mid 1\leq j\leq m, \ j \neq i \} $$
    Hence the vertices of type $x_i:y_j$ in $A_i$ are located in the union of
     two fibers, namely in $(\{x_i\}\times K_{1,m})\cup (K_{1,n}\times \{y_i\})$,
     omitting their intersection vertex $x_i:y_i$.
  \item 2 colors for the edges ending in the diagonal vertices $x_i:y_i$, namely
    one color for $\{ (x_i:y,x_i:y_i) \mid 1\leq i\leq m \}$
     and one for $\{ (x:y_i,x_i:y_i) \mid 1\leq i\leq m \}$.
\end{itemize}
It is easily checked that each color class is an induced matching, and their
 union covers the entire edge set.

%We first color all the edges not incident to the vertex $x:y$. The collection of all the vertices which are not adjacent to $x:y$ is $\{x_i:y_j\mid i \in [n], j \in [m]\}$. %We first color all the edges not incident to the vertex $x:a$, that is the graph induced by $\{x_i:a_j\mid i \in [n], j \in [m]\}$.

%For any $i \in [n]$ and $j \in [m]$, $x_i:y$ and $x:y_j$ are the only neighbors of $x_i:y_j$.
% So for any $i \in [n], \{ (x_i:y_j,x:y_j)\mid  1\leq j\leq m\}$ and for $j \in [m], \{ (x_i:y_j,x_i:y)\}\mid  1\leq i\leq n\}$ are induced matchings.
%Again $x:a_j)$ is adjacent to $x_i:a_{j'}$ if and only if $j = j'$, also $x_i:a_j$ is adjacent to $x_{i'}:a$ if and only if $i = i'$.
%Also for every $j \in \{1,2,...,m\},
%  A_j := \{ (x_j:y_{j'},x:y_{j'})\mid  j' \neq j$ and $1\leq i\leq m\} \cup \{ (x_i:y_j,x_i:y)\mid  i > j, 1\leq i \leq n\}$, for $m <j \leq n$, define $A_j := \{ (x_j:y_{j'},x:y_{j'})\mid  1\leq j' \leq m\}$. All these sets are induced matchings.
% Also define $A_{n+1} := \{ (x_j:y_j,x:y_j)\mid  1\leq j\leq m\}$ and $A_{n+2} := \{ (x_j:y_j,x_j:y)\mid  1\leq j\leq m\}$ 

%Again $E\setminus (\bigcup^{n+2}_{j=1} A_j)$ is the collection of all the edges incident to the vertex $x:y$
%and degree of $x:y$ is $n+m$.
%\{((0,0),(a,0)),(0,0_i)\mid  1\leq a\leq n, 1\leq i\leq m\}.$$ 
%Therefore $\chi _s'(K_{1,n} \Box K_{1,m}) \leq 2n+m+2$.

Next we show that fewer than $2n+m+2$ colors are not enough.
Note that each of the $n+m$ edges incident with the vertex $x:y$ is adjacent to every other edge.
This requires $n+m$ private colors for these edges. Hence it is sufficient to prove that the edges not incident with $x:y$ require more than $n+1$ further colors.
Since no color can occur more than once on the edges of any $K_{1,m}$-fiber centered at an $x_i:y$, neither on any $K_{1,n}$-fiber centered at an $x:y_i$, only the following three types of color classes can occur:
\begin{enumerate}[label=(\roman*)]
  \item Precisely $m$ edges $(x:y_i,x_{j_i}:y_i)$ for $i=1,2,\dots,m$, where $j_i$ is arbitrary ($1\le j_i\le n$);
  let the number of such color classes be denoted by $a$.
  \item At most $m-1$ edges incident with vertices of type $x:y_i$, and at most $n-1$ edges incident with some of the $x_i:y$\,;
     let the number of such color classes be denoted by $b$.
  \item Precisely $n$ edges $(x_i:y,x_i:y_{j_i})$ for $i=1,2,\dots,n$, where $j_i$ is arbitrary ($1\le j_i\le m$);
     let the number of such color classes be denoted by $c$.
\end{enumerate}
Hence the number of edges involved in such a color partition is
 distributed as shown in Table \ref{tab:star}.
(The row correspondng to $b$ gives an upper bound, the other lines are exact.)

\renewcommand{\arraystretch}{1.2}
  \begin{table}[htp]
	\begin{center}%
\begin{tabular}
[c]{|c||c|c|}\hline
\# of classes & $K_{1,m}$-edges & $K_{1,n}$-edges  \\ \hline\hline
$a$ & $m$ & 0 \\
$b$ & $m-1$ & $n-1$ \\ %\hline
$c$ & 0 & $n$ \\ %\hline
 \hline
\end{tabular}
 \caption{Number of classes, and largest possible number of edges per class.
   \label{tab:star}}
\end{center}
  \end{table}

The $n$ copies of $K_{1,m}$, as well as the $m$ copies of $K_{1,n}$,
 contain precisely $mn$ edges.
Therefore the following inequalities must hold:
 \begin{eqnarray}
   am + b(m-1) & \ge & mn \\
   b(n-1) + cn & \ge & mn
 \end{eqnarray}
Rearrangement yields
 \begin{equation}   \label{a+b}
   a+b \ge n + b/m \qquad \mbox{\rm and} \qquad c \ge m - b + b/n 
 \end{equation}
 Thus, 
  $$ a+b+c \ge n + m - b (1 - 1/m - 1/n). $$
This inequality implies the required lower bound $a+b+c>n+1$, unless
 $$ b \left( 1 - \frac{1}{m} - \frac{1}{n} \right) \ge m-1. $$
In this case, however, we have
 $$ b \ge \frac{m-1}{1-\frac{1}{m}-\frac{1}{n}}
    > \frac{m-1}{1-\frac{1}{m}} = m, $$
  from which, using (\ref{a+b}), we obtain
 $$ a + b \ge n + b/m > n + 1. $$
 
 Therefore we need at least $2n+m+2$ colors. Hence $\sci(K_{1,n} \Box K_{1,m}) = 2n+m+2$.
\end{proof}

% In the final section, we consider the product of a tree with a cycle.

\section{Cartesian product of a tree and a cycle}
In this section we give lower and upper bounds for the strong chromatic index of the Cartesian product of a tree and a cycle. Also we determine the exact value of that, if the length of the cycle is $4\ell$ for some $\ell \geq 1$.

Throughout this section $T$ means any tree, and
 $\Delta$ stands for its maximum degree, as a shorthand for $\Delta(T)$.

\begin{lemma}\label{cycle-lower}
Let $C$ be any cycle of length at least $4$. Then $\sci(T \Box C) \geq 2\Delta+4$.
\end{lemma}

\begin{proof}

Corollary \ref{cor:path} says that the strong chromatic index of the Cartesian product of $T$ and a path longer than one is $2\Delta+4$. Now $C$ contains a path with $3$ vertices. Therefore $\sci(T \Box C) \geq 2\Delta+4$.
\end{proof}

This inequality holds with equality if the length of the cycle is divisible by $4$. We prove this in the next theorem.

\begin{theorem}\label{tree-tree}
The strong chromatic index of $T \Box C_{4\ell}$ is $2\Delta+4$.
\end{theorem}

\begin{proof}
The above lemma says that $\sci(T \Box C_{4\ell}) \geq 2\Delta+4$. Now we show that there is a  strong edge coloring on  $T \Box C_{4\ell}$  that uses $2\Delta+4$ colors. 

%Moreover, it satisfies the following properties
%\begin{enumerate}[label=(\roman*)]
%\item For all $(x,x') \in E(T)$ and $y \in V(C),1 \leq$ the color of $(x:y,x':y) \leq 2\Delta(T)$
%\item For all $x \in V(T)$ and $(y,y') \in E(C), 2\Delta(T)+1 \leq (x:y,x:y') \leq 2\Delta(T)+4$.
%\end{enumerate}

Let $a$ be a vertex in $T$ and set $C_{4\ell} = t_1t_2...t_{4\ell}t_1$.
The proof of Theorem \ref{ub} says that there is a strong edge coloring $c$ on $T \Box (C_{4\ell} \setminus \{(t_{4\ell},t_1)\})$ with colors $1,2,3,...,2\Delta+4$ such that $c$ satisfies the following properties:

\begin{enumerate}[label=(\roman*)]
\item\label{cira} All the edges of $T \times V(C_{4\ell})$ use color from $\{1,2,3,...,2\Delta\}$.
%If $(u,u') \in E(T_1)$ and $v \in T_2$, color of the edge $(u:v,u':v)$ is one of the color among $1,2,3,...,2\Delta(T_1)$.
\item\label{cirb} All the edges of $V(T) \times C_{4\ell}$ use color from $\{2\Delta+1, 2\Delta+2, 2\Delta+3, 2\Delta+4\}$.
\item\label{circ} For every $(u,u')\in E(T), i,j \in [4\ell-2]$, if $2|(i-j)$, then $c((u:t_i,u':t_i)) = c((u:t_j,u':t_j))$, .
\item\label{cird} If the distance from $u$ to $u'$ is even in $T$, then $c((u:t_i,u:t_{i+1})) = c((u':t_i,u':t_{i+1}))$, for every $i \in [4\ell-1]$.
\end{enumerate}

Let $b$ be a vertex adjacent to $a$ in $T$.
The definition of strong edge coloring ensures that the colors of the edges $(a:t_1,a:t_2),(a:t_2,a:t_3),(b:t_1,b:t_2)$ and $(b:t_2,b:t_3)$ are distinct.
Again according to \ref{cirb} they are $2\Delta+1, 2\Delta+2, 2\Delta+3, 2\Delta+4$. Without loss of generality we can assume that
$c((a:t_1,a:t_2)) = 2\Delta+1, c((a:t_2,a:t_3)) = 2\Delta+2, c((b:t_1,b:t_2)) = 2\Delta+3$ and $c((b:t_2,b:t_3)) = 2\Delta+4$.

%\begin{figure}
%\begin{center}
%\begin{tikzpicture}

%\draw[fill](0,0)circle[radius=0.1]node[black,below=4]{$a:t_1$};
%\draw[fill](4,0)circle[radius=0.1]node[black,below=4]{$a:t_2$};
%\draw[fill](8,0)circle[radius=0.1]node[black,below=4]{$a:t_3$};
%\draw[fill](0,2)circle[radius=0.1]node[black,above=4]{$b:t_1$};
%\draw[fill](4,2)circle[radius=0.1]node[black,above=4]{$b:t_2$};
%\draw[fill](8,2)circle[radius=0.1]node[black,above=4]{$b:t_3$};

%\draw [thick] (0,0)--(4,0)
% node[pos=0.5,below]{$2\Delta+1$};
%\draw [thick] (4,0)--(8,0)
% node[pos=0.5,below]{$2\Delta+2$};
%\draw [thick] (0,2)--(5,2)
% node[pos=0.5,below]{$2\Delta+3$};
%\draw [thick] (4,2)--(8,2)
% node[pos=0.5,below]{$2\Delta_1+4$};
 
%\draw [thick] (8,0)--(10,0);
%\draw [thick] (8,2)--(10,2);
%\draw [thick] (0,2)--(0,0);
%\draw [thick] (4,2)--(4,0);
%\draw [thick] (8,2)--(8,0);
%\end{tikzpicture}
%\end{center}
%\end{figure}

Now the color of $(a:t_3,a:t_4)$ cannot be that of $(a:t_1,a:t_2),(a:t_2,a:t_3)$ and $(b:t_2,b:t_3)$. Hence it is $2\Delta+3$.
Repeating this argument again and again, together with \ref{cird} it leads to the following fact:
\begin{fact}\label{fact1}
 $c((u:t_i,u:t_{i+1})) = 2\Delta+1+( i-1 \bmod 4)$ if $d_{T}(u,a)$ is even, otherwise it is $2\Delta+1+( i-3 \bmod 4) $.
\end{fact}
%\begin{align}\label{fact}
%c((u:t_i,u:t_{i+1})) &= 
%\begin{cases}
%2\Delta(T)+1, &\text{ \ \ \ \ if } d_{T}(u,a) \text{ is even, } i \equiv 1 \mod(4)\\
% &\text{ or, if } d_{T}(u,a) \text{ is odd, } i \equiv 3 \mod(4)\\
%2\Delta(T)+3, &\text{ \ \ \ \ if } d_{T}(u,a) \text{ is odd, } i \equiv 1 \mod(4)\\
% &\text{ or, if } d_{T}(u,a) \text{ is even, } i \equiv 3 \mod(4)\\
%2\Delta(T)+2, &\text{ \ \ \ \ if } d_{T}(u,a) \text{ is even, } i \equiv 2 \mod(4)\\
% &\text{ or, if } d_{T}(u,a) \text{ is odd, } i \equiv 4 \mod(4)\\
%2\Delta(T)+4, &\text{ \ \ \ \ if } d_{T}(u,a) \text{ is odd, } i \equiv 2 \mod(4)\\
% &\text{ or, if } d_{T}(u,a) \text{ is even, } i \equiv 4 \mod(4)\\
%\end{cases}
%\end{align}

%Therefore $c((u:t_1,u:t_2)) \neq c((u:t_{4\ell-1},u:t_{4\ell}))$ and by \ref{}, $c((u:t_1,u':t_1)) = c((u:t_{4\ell-1},u':t{4\ell-1}))$, for every adjacent vertex $u'$ of $u$. 
We can extend this coloring to  $T \Box C_{4\ell}$ by defining

$$c((u:t_{4\ell},u:t_1)) :=
\begin{cases}
2\Delta+2, \,\text{ if } d_T (u,a) \text{ is odd }\\
2\Delta+4, \,\text{ otherwise }\\
\end{cases}$$

We claim that this is a strong edge coloring.
Since the coloring $c$ satisfies \ref{cira}, \ref{cirb},  \ref{circ} and \ref{cird}, it is enough to show the following three properties for every vertex $u$ of $T$ and any two  vertices $v$ and $v'$ adjacent to it.
\begin{enumerate}[label=(\alph*)]
\item\label{e1} $c((u:t_1,v:t_1)) \neq c((u:t_{4\ell},v':t_{4\ell}))$.
\item\label{e2} $c((u:t_1,u:t_2)) \not\in \{ c((u:t_{4\ell},u:t_1)),c((v:t_{4\ell},v:t_1)),
c((u:t_{4\ell-1},u:t_{4\ell}))\}$
\item\label{e3} $c((u:t_{4\ell},u:t_1)) \not\in \{  c((u:t_{4\ell-2},u:t_{4\ell-1})), c((u:t_2,u:t_3)), c((v:t_{4\ell},v:t_1))\}$
\end{enumerate}

Now \ref{circ} says that $c((u:t_1,v:t_1)) = c((u:t_{4\ell-1},v:t_{4\ell-1}))$.
Again $c$ is a strong edge coloring on $T \Box (C_{4\ell} \setminus \{(t_{4\ell},t_1)\}$.
Hence \ref{e1} holds true.
Moreover,  Fact~\ref{fact1} says that \ref{e2} and \ref{e3} are valid, too.
%Now we define the color function $c'$ from the edge set of $T \Box C_{4\ell}$ to $\{1,2,...,2\Delta(T)+4\}$ as follows;
%\begin{align*}
%c'((v:t_i,v':t_i)) &= c((v:t_i,v':t_i))\\
%c((v:t_{4i+1},v:t_{4i+2})) &= 
%\begin{cases}
% 2\Delta(T)+1, \textmd{ if } d_T(v,u) \text{ is even }\\
% 2\Delta(T)+3, \textmd{ if } d_T(v,u) \text{ is odd}\\
%\end{cases}\\
%c((v:t_{4i+2},v:t_{4i+3})) &= 
%\begin{cases}
% 2\Delta(T)+2, \textmd{ if } d_T(v,u) \text{ is even }\\
% 2\Delta(T)+4, \textmd{ if } d_T(v,u) \text{ is odd}\\
%\end{cases}\\
%c((v:t_{4i+3},v:t_{4i+4})) &=
%\begin{cases}
% 2\Delta(T)+3, \textmd{ if } d_T(v,u) \text{ is even }\\
% 2\Delta(T)+1, \textmd{ if } d_T(v,u) \text{ is odd}\\
%\end{cases}\\
%c((v:t_{4i+4},v:t_{4i+1})) &=
%\begin{cases}
% 2\Delta(T)+4, \textmd{ if } d_T(v,u) \text{ is even }\\
% 2\Delta(T)+2, \textmd{ if } d_T(v,u) \text{ is odd}\\
%\end{cases}\\
%\end{align*}
\end{proof}

The above theorem determines the strong chromatic index of the Cartesian product of a tree and a cycle whose length is a multiple of $4$. 
Next we find an upper bound for the Cartesian product of a tree with a cycle of any even length. 
We state this for the $6$-cycle first, and then in Theorem \ref{largeevencycle} we find a better bound for the remaining even cycles.

\begin{theorem}   \label{t:cycle6}
The strong chromatic index of $T \Box C_6$ is at most $2\Delta(T)+6$.
\end{theorem}

\begin{proof}
Let $a$ be a vertex in $T$, and set $C_6 = t_1t_2t_3t_4t_5t_6t_1$.
Consider the strong edge coloring $c$ on $T \Box (C_6 \setminus \{(t_6,t_1)\})$ as described in Theorem \ref{tree-tree}.
It follows that the coloring $c$ satisfies the following properties:
\begin{enumerate}[label=(\roman*)]
\item All the the edges of $T \times V(C_6)$ use colors from $\{1,2,3,...,2\Delta\}$.
\item\label{i-i+2} $c((u:t_i,u':t_i)) = c((u:t_j,u':t_j))$, whenever $2|(i-j)$.
\end{enumerate}

We now extend the coloring to obtain the color function $c'$ on the edges of $T\Box C_6$ as follows. Let $c'((u:t_i,v:t_i))=c((u:t_i,v:t_i))$, for all $i$ and $(u,v) \in E(T)$. For every vertex $u$ in $T$, define  $c'((u:t_i,u:t_{i+1}))$ to be $2\Delta+1+(i \mod 3)$ if $d_T(u,a)$ is even, else let $2\Delta+4+(i \mod 3)$.

Thus for any two adjacent vertices $u$ and $v$ in $T$, we see $c'((u:t_i,u:t_{i+1})) \not\in \{c'((v:t_i,v:t_{i+1})), c'((u:t_{i+1},u:t_{i+2})),c'((v:t_{i+1},v:t_{i+2})),
c'((u:t_{i+2},u:t_{i+3}))\}$.
Again $c$ is a strong edge coloring and it satisfies \ref{i-i+2}. So $c'((u:t_i,v:t_i)) \neq c'((u:t_j,v:t_j))$, for all $t_i$ adjacent to $t_j$. Hence $c'$ is a strong edge coloring.  
\end{proof}

This bound is tight. The strong chromatic bound of $C_6 \Box P_2$ is $8$. On the other hand there are also graphs whose product does not attend this bound. The strong chromatic index of $C_6 \Box P_3$ is $9$. 

\begin{theorem}\label{largeevencycle}
Let $\ell >3$. Then $\sci(T \Box C_{2\ell}) \leq 2\Delta(T)+5$.
\end{theorem}

\begin{proof}
Let $a$ be a vertex in $T$, and set $C_{2\ell} = t_1t_2...t_{2\ell}t_1$.
If $\ell$ is even, then Theorem \ref{tree-tree} says that the stated bound is true. Now we may assume that $\ell$ is odd. 
In the proof of Theorem \ref{ub} we showed that there is a strong edge coloring $c$ on $T \Box (C_{2\ell} \setminus \{(t_{2\ell},t_1)\})$ with colors $1,2,3,...,2\Delta+4$ such that $c$ satisfies the following properties:

\begin{enumerate}[label=(\roman*)]
\item\label{cira2} All the edges from every $T$-fiber use colors from the set $\{1,2,3,...,2\Delta\}$.
\item\label{cirb2} All the edges from every $C$-fiber use colors from the set $\{2\Delta+1, 2\Delta+2, 2\Delta+3, 2\Delta+4\}$.
\item\label{circ2} $c((u:t_i,u':t_i)) = c((u:t_{i+2},u':t_{i+2}))$, for every $(u,u')\in E(T), i \leq 2\ell-2$.
\item\label{cird2} If the distance from $u$ to $u'$ is even in $T$, then $c((u:t_i,u:t_{i+1})) = c((u':t_i,u':t_{i+1}))$, for every $i \in [2\ell-1]$.
\end{enumerate}

Define the color function $c'$ on $T \Box V(C_{2\ell})$ as that of $c$. Now \ref{circ2} says that $c'((u:t_1,u':t_1)) = c'((u:t_{2\ell-1},u':t_{2\ell-1}))$. Hence $c'((u:t_1,u':t_1)) \neq c'((u:t_{2\ell},v:t_{2\ell}))$ and $c'((u:t_{2\ell},v:t_{2\ell})) \neq c'((u:t_{2\ell-1},v:t_{2\ell-1}))$, for any vertex $v$  adjacent to $u$ (including $u'$). Therefore $c'$ is a partial coloring on $T \Box C_{2\ell}$ that uses $2\Delta$ colors. Now we extend this coloring by using $5$ new colors.

Now the only edges which remain to be colored are the edges of $C$-fibers. We partition them into five induced matchings. 
For every $i\in [2\ell]$ define the sets $A_i := \{ (u:t_i,u:t_{i+1})\mid  d_T(u,a)$ is even$\}$ and 
$B_i := \{ (u:t_i,u:t_{i+1})\mid  d_T(u,a)$ is odd$\}$. Clearly every uncolored edges appears exactly in one of the $A_i$ or $B_i$. 
Let $(u:t_i,u:t_{i+1})$ and $(v:t_i,v:t_{i+1})$ be two edges of $A_i$. 
By the definition of the Cartesian product, $u:t_i$ is not adjacent to $v:t_{i+1}$ and $u:t_{i+1}$ is not adjacent to $v:t_i$. 
Again both $d_T(u,a)$ and $d_T(v,a)$ are even. So $u:t_i$ (respectively $u:t_{i+1}$) is not adjacent to $v:t_i$ (respectively $v:t_{i+1}$). Therefore $A_i$ is an induced matching, for every $i$. Similarly we can show that for every $i$, $B_i$ is an induced matching. Moreover it satisfies the following properties.
\begin{enumerate}
\item For $|i-j|>2 \mod(2\ell)$ the sets $A_i \cup A_j$ and $B_i \cup B_j$ are induced matchings.

\textit{Proof:} Let $(u:t_i,u:t_{i+1}) \in A_i$ and $(v:t_j,v:t_{j+1}) \in A_j$. Since $|i-j|>2 \mod(2\ell)$, the definition of the strong edge coloring says that $\{ (u:t_i,u:t_{i+1}),(v:t_j,v:t_{j+1})\}$ is an induced matching. 
Again both $A_i$ and $A_j$ are induced matchings. Hence so is $A_i \cup A_j$. Similarly we can show that $B_i \cup B_j$ is an induced matching. 
\item Let $t_i$ be not adjacent to $t_j$. Then $A_i \cup B_j$ and $A_j \cup B_i$ are induced matchings.

The proof is similar to that of $1$.
\end{enumerate}

We have assumed that $i>3$ and $2\ell \equiv 2 \mod 4$. Therefore the above mentioned properties say that the following five sets are induced matchings:
$$A_2 \cup B_4, \qquad
(\bigcup_{i \equiv 1 \mod 4} A_i) \cup (\bigcup_{i \equiv 3 \mod 4} B_i), \qquad
(\bigcup_{i \equiv 3 \mod 4} A_i) \cup (\bigcup_{i \equiv 1 \mod 4} B_i),$$

$$B_2 \cup (\bigcup_{\substack{i \geq 6 \\ i \equiv 2 \mod 4}} A_i) \cup (\bigcup_{\substack{i > 6 \\ i \equiv 0 \mod 4}} B_i), \qquad
(\bigcup_{\substack{i \geq 4 \\ i \equiv 0 \mod 4}} A_i) \cup (\bigcup_{\substack{i > 4 \\ i \equiv 2 \mod 4}} B_i).$$
\end{proof}

This bound is also tight. The strong chromatic index of $C_{10} \Box P_2$ is $7$.
The above theorem and Theorem \ref{cycle-lower} say that the strong chromatic index of the Cartesian product of a tree $T$ and an even cycle $C_{2\ell}$ of length more than $6$ is either $2\Delta(T)+4$ or $2\Delta(T)+5$.
In the next two theorems we prove an upper bound for the strong chromatic index of any tree with cycles of odd length. First we consider cycles of length more than $8$.

\begin{theorem}   \label{t:longodd}
Let $\ell >3$. Then $\sci(T \Box C_{2\ell+1}) \leq 2\Delta(T)+\lceil \frac{\Delta(T)}{\ell} \rceil +5$.
\end{theorem}

\begin{proof}
Let $a$ be a vertex of maximum degree in $T$, and set $C_{2\ell+1} = t_1t_2...t_{2\ell+1}t_1$.
We view $a$ as the root of $T$.
Let $D_1,D_2,...,D_{\sci(T)}$  be a partition (some of the $D_i$ may be empty) of the edges of $T$ into induced matchings.
We are going to define a strong edge coloring $c$ on  $T \Box C_{2\ell+1}$. 
First we give colors to the edges of all the $T$-fibers. 

For each $1 \leq i \leq \lfloor \Delta/2 \rfloor$, define the following two $\Delta$-tuples:
$c_{2i+1} := (\Delta+1,\Delta+2,...,\Delta+(i-1)\lceil \frac{\Delta}{\ell} \rceil, 2\Delta+1,2\Delta+2,...,2\Delta+\lceil \frac{\Delta}{\ell} \rceil, 
i\lceil \frac{\Delta}{\ell} \rceil+1,i\lceil \frac{\Delta}{\ell} \rceil+2,...,\Delta)$ and
$c_{2i} := (1,2,...,(i-1)\lceil \frac{\Delta}{\ell} \rceil, \Delta+(i-1)\lceil \frac{\Delta}{\ell} \rceil+1,...,2\Delta)$. That is the $j^{th}$ term of $c_{2i+1}$,
\begin{equation*}
	(c_{2i+1})_j =
\begin{cases}
	\Delta+j, 1 \leq j \leq (i-1)\lceil \frac{\Delta}{\ell} \rceil\\
	2\Delta+j-i+1, (i-1)\lceil \frac{\Delta}{\ell} \rceil < j \leq i \lceil \frac{\Delta}{\ell} \rceil\\
	j, i\lceil \frac{\Delta}{\ell} \rceil < j \leq \Delta\\
\end{cases}
\end{equation*}
and the $j^{th}$ term of $c_{2i}$, 
\begin{equation*}
	(c_{2i})_j =
	\begin{cases}
		j, 1 \leq j \leq (i-1)\lceil \frac{\Delta}{\ell} \rceil\\
		\Delta+j, (i-1)\lceil \frac{\Delta}{\ell} \rceil < j \leq \Delta\\
	\end{cases}
\end{equation*}

Note that $(c_{i+1})_{j'} \neq (c_i)_j$ and $(c_i)_j \neq (c_i)_{j'}$, for all $i,j,j'$.

Now for any vertex $u'$ and its parent $u$ in $T$, define

$$c((u:t_i,u':t_i)) :=
\begin{cases}
(c_i)_j, \ \ \ \,\text{ if } d_T(a,u) \text{ is even  and } (u,u') \in D_j\\
(c_{i+1})_j, \ \text{ if } d_T(a,u) \text{ is odd  and } (u,u') \in D_j\\
\end{cases}$$

We claim that each color class is an induced matching in $T \Box C_{2\ell+1}$.

Let $(u:t_i,u':t_i)$ and $(u:t_i,u'':t_i)$ be two adjacent edges.
Then, since every $D_i$ is an induced matching, both $(u,u')$ and $(u,u'')$ cannot be in the same $D_i$.
Hence $c((u:t_i,u':t_i)) \neq c((u:t_i,u'':t_i))$.
Now let $v:t_k$ be adjacent to $u:t_i$.
The definition of the Cartesian product says that either $i=k$ or $u=v$.
If $i=k$, then  $c((u:t_i,u':t_i)) \neq c((v:t_i,v':t_i))$, for any $v'$ (using  the properties of the sets $c_j$, and that the $D_i$ are induced matchings). 
If $i\neq k$, then $u=v$ and $t_i$ is adjacent to $t_j$.
Without loss of generality, let $k = i+1 \mod(2\ell+1)$.
 We know that $(c_i)_j \neq (c_{i+1})_{j'}$, for all $i,j,j'$. So $c((u:t_i,u':t_i)) \neq c((v:t_i,v':t_i))$.
 Therefore every color class is an induced matching.

The remaining (uncolored) edges are the edges of all the $C$-fibers. Therefore it is sufficient to partition them into five induced matchings. As in the proof of the earlier theorem, we partition them into the sets 
$A_i := \{ (u:t_i,u:t_{i+1})\mid  d_T(u,a)$ is even$\}$ and 
$B_i := \{ (u:t_i,u:t_{i+1})\mid  d_T(u,a)$ is odd$\}$, $i\in [2\ell+1]$.
It is easy to see that the following properties hold.
\begin{enumerate}[label=(\roman*)]
\item For every $i$, $A_i$ and $B_i$ are induced matchings.
\item Let $|i-j|>2 \mod(2\ell+1)$. Then $A_i \cup A_j$ and $B_i \cup B_j$ are induced matchings.
\item Let $t_i$ is not adjacent to $j$. Then $A_i \cup B_j$ and $A_j \cup B_i$ are induced matchings.
\end{enumerate}
 
Now $B_1 \cup A_5 \cup B_7$ is an induced matching because $2\ell+1 \equiv 1$ or $3\mod 4$.
Therefore it is enough to partition the rest of the edges into four induced matchings.
If $2\ell+1 \equiv 1 \mod 4$, then the sets
 $$A_1 \cup B_3 \cup (\bigcup \limits_{\substack{i > 6 \\ i\equiv 0 \mod 4}} A_i) \cup (\bigcup \limits_{\substack{i \geq 6 \\ i\equiv 2 \mod 4)}} B_i), \qquad
  (\bigcup \limits_{i\equiv 2 \mod 4} A_i) \cup (\bigcup \limits_{i\equiv 0 \mod 4} B_i),$$
$$B_2 \cup A_4 \cup (\bigcup \limits_{\substack{i \geq 7 \\ i\equiv 3 \mod 4}} A_i) \ \ \cup (\bigcup \limits_{\substack{i > 7 \\ i\equiv 1 \mod 4}} B_i), \qquad
A_3 \cup B_5 \cup  (\bigcup \limits_{\substack{i \geq 9 \\ i\equiv 1 \mod 4}} A_i) \cup (\bigcup \limits_{\substack{i > 9 \\ i\equiv 3 \mod 4}} B_i)$$
 partition $V(T) \times C_{2\ell+1}$ into induced matchings. 
Otherwise $2\ell+1 \equiv 3 \mod 4$, and then the sets
 $$B_2 \cup A_4 \cup (\bigcup \limits_{\substack{i > 6 \\ i\equiv 0 \mod 4}} A_i) \cup (\bigcup\limits _{\substack{i \geq 6 \\ i\equiv 2 \mod 4)}} B_i), \qquad
( \bigcup \limits_{i\equiv 2 \mod 4} A_i) \cup (\bigcup \limits_{i\equiv 0 \mod 4} B_i),$$ $$A_2 \cup B_4 \cup (\bigcup \limits_{\substack{i \geq 7 \\ i\equiv 3 \mod 4}} A_i) \cup (\bigcup \limits_{\substack{i > 7 \\ i\equiv 1 \mod 4}} B_i), \qquad
A_3 \cup B_5 \cup (\bigcup \limits_{\substack{i \geq 9 \\ i\equiv 1 \mod 4}} A_i) \cup (\bigcup \limits_{\substack{i > 9 \\ i\equiv 3 \mod 4}} B_i)$$
 form a required partition.

Therefore the strong chromatic index of $T \Box C_{2\ell+1}$  is at most $2\Delta+\lceil \frac{\Delta}{\ell} \rceil +5$, for $\ell>3$.
%$$S_2 := 
%\begin{cases}
%A_1 \cup B_3 \cup (\bigcup \limits_{\substack{i > 6 \\ i\equiv 0 \mod 4}} A_i) \cup (\bigcup \limits_{\substack{i \geq 6 \\ i\equiv 2 \mod 4)}} B_i), \ \ 2\ell+1 \equiv 1\mod 4\\
%B_2 \cup A_4 \cup (\bigcup \limits_{\substack{i > 6 \\ i\equiv 0 \mod 4}} A_i) \cup (\bigcup\limits _{\substack{i \geq 6 \\ i\equiv 2 \mod 4)}} B_i), \ \ 2\ell+1 \equiv 3\mod 4\\
%\end{cases}$$
%
%$$S_3 := 
%\begin{cases}
%\bigcup \limits_{i\equiv 2 \mod 4} A_i \cup (\bigcup \limits_{i\equiv 0 \mod 4} B_i), \ \ 2\ell+1 \equiv 1\mod 4\\
%\bigcup \limits_{i\equiv 0 \mod 4} A_i \cup (\bigcup \limits_{i\equiv 2 \mod 4} B_i),\ \ 2\ell+1 \equiv 3\mod 4\\
%\end{cases}$$
%
%$$S_4 := 
%\begin{cases}
%B_2 \cup A_4 \cup (\bigcup \limits_{\substack{i \geq 7 \\ i\equiv 3 \mod 4}} A_i) \cup (\bigcup \limits_{\substack{i > 7 \\ i\equiv 1 \mod 4}} B_i), \ \ 2\ell+1 \equiv 1\mod 4\\
%A_2 \cup B_4 \cup (\bigcup \limits_{\substack{i \geq 7 \\ i\equiv 3 \mod 4}} A_i) \cup (\bigcup \limits_{\substack{i > 7 \\ i\equiv 1 \mod 4}} B_i), \ \ 2\ell+1 \equiv 3\mod 4\\
%\end{cases}$$
%
%$$S_5 := 
%\begin{cases}
%A_3 \cup B_5 \cup  (\bigcup \limits_{\substack{i \geq 9 \\ i\equiv 4 \mod 4}} A_i) \cup (\bigcup \limits_{\substack{i > 9 \\ i\equiv 2 \mod 4}} B_i)
%\ \ 2\ell+1 \equiv 1\mod 4\\
%A_3 \cup B_5 \cup (\bigcup \limits_{\substack{i \geq 9 \\ i\equiv 3 \mod 4}} A_i) \cup (\bigcup \limits_{\substack{i > 9 \\ i\equiv 1 \mod 4}} B_i), \ \ 2\ell+1 \equiv 3\mod 4\\
%\end{cases}$$
\end{proof}

Above we obtained bounds for the strong chromatic index of a tree with a cycle of length different from $3,5$ and $7$. Now we find a slightly larger upper bound for the same when a cycle is of length $3,5$ and $7$.

\begin{theorem}   \label{t:shortodd}
Let $\ell \leq 3$. Then $\sci(T \Box C_{2\ell+1}) \leq  2\Delta(T)+\lceil \frac{\Delta(T)}{\ell} \rceil +6$.
\end{theorem}

\begin{proof}
Let $a$ be an arbitrary vertex of $T$.
We can assign the colors $1,2,3,...,$ $2\Delta+\lceil \frac{\Delta}{\ell} \rceil$ to the edges of $T \times V(C_{2\ell+1})$ in the same way as in the proof of the previous theorem. The only edges which remain to be colored are the edges of $V(T) \times C_{2\ell+1}$.

Define $A_i$ and $B_i$ as in the preceding proof, for $i \leq 2\ell+1$. 
We have already shown above that the sets $A_i$ and $B_i$ are induced matchings.
Hence $\sci(T \Box C_3) \leq 3\Delta+6$.
Now further properties of these $A_i$ and $B_i$ (as mentioned above) ensure that the following properties are true.

If $\ell =2$, then $A_1 \cup B_3$, $B_1 \cup A_3$, $A_2 \cup B_4$, $B_2 \cup A_4$, $A_5$, and $B_5$ are induced matchings;
and if $\ell =3$, then $A_1 \cup B_3 \cup A_5$, $B_1 \cup A_3 \cup B_5$, $A_2 \cup B_4 \cup A_6$, $B_2 \cup A_4 \cup B_6$, $A_7$, and $B_7$ are induced matchings. 
Therefore $\sci(T \Box C_{2\ell+1}) \leq  
2\Delta+\lceil \frac{\Delta}{\ell} \rceil +6$, for $\ell \leq 3$.
\end{proof}

Finally, we give a better lower bound for the same when the cycle is of odd length.
For that, first we define jellyfish graphs.
Let $H$ be a graph obtained from $C_k$ by adding $p_v$ new pendant vertices adjacent to $v$, for each vertex $v$ in $C_k$. Then $H$ is called a $C_k$-{\it jellyfish}.
On such graphs, Chang et al. \cite{chang2015strong} proved a general lower bound (their Theorem $13$), which implies the following formula.

\begin{theorem}\cite{chang2015strong}
 If $G$ is a $C_k$-jellyfish of $m$ edges, such that $k$ is odd and all vertices of $C_k$ have the same degree in $G$, then 
 $\sci(G) \geq \lceil \frac{m}{\lfloor k/2 \rfloor}\rceil$.
\end{theorem}

By using this theorem, we can derive an improved lower bound for the strong chromatic index of the product of a tree with an odd cycle.

\begin{corollary}   \label{cor:LB-odd}
Let $T$ be a tree. Then $\sci(T \Box C_{2\ell+1}) \geq \lceil \frac{(2\ell+1)(\Delta(T)+1)}{\ell}\rceil
$.
\end{corollary}

\begin{proof}
 Let $u$ be a vertex of maximum degree in $T$, and set $C_{2\ell+1} = t_1t_2...t_{2\ell+1}t_1$.
 Now $T \Box C_{2\ell+1}$ contains $G[N(u) \cup \{u\}] \Box C_{2\ell+1}$, which contains a $(\Delta+2)$-regular 
 $C_{2\ell+1}$-jellyfish. Therefore $\sci(T \Box C_{2\ell+1}) \geq \lceil \frac{(2\ell+1)(\Delta+1)}{\ell}\rceil$.
\end{proof}

\begin{remark}
Note that the difference of the lower and the upper bound for the strong chromatic index of a tree with an odd cycle is at most $4$. In fact, if the odd cycle is of length at least $9$, then the difference is at most $3$.
\end{remark}

\begin{proof}
	\begin{align*}
		\sci(T \Box C_{2\ell+1}) & \geq \lceil \frac{(2\ell+1)(\Delta+1)}{\ell}\rceil \\
		& = \lceil 2\ell \frac{\Delta(T)+1}{\ell} + \frac{\Delta(T)+1}{\ell} \rceil \\
		& \geq 2\Delta +2 +\lceil \frac{\Delta(T)}{\ell} \rceil.
	\end{align*}
\end{proof}	
	 
\section{Conclusion}

Table \ref{tab:cycle} summarizes the results which involve cycles.

\noindent
\renewcommand{\arraystretch}{1.2}
  \begin{table}[htp]
	\begin{center}%
\begin{tabular}
[c]{|c||c|c||c|c|}\hline
cycle length $k=$ & upper bound & theorem & lower bound & result \\ \hline\hline
$4\ell$ & $2\Delta+4$ & \ref{tree-tree} & $2\Delta+4$ & Lemma \ref{cycle-lower} \\
6 & $2\Delta+6$ & \ref{t:cycle6} & & \\
$4\ell+2\geq 10$ & $2\Delta+5$ & \ref{largeevencycle} & & \\ \hline
$2\ell+1\leq 7$ & $2\Delta+\lceil \frac{\Delta}{\ell} \rceil+6$ & \ref{t:shortodd} & $2\Delta+\lceil \frac{\Delta+1}{\ell} \rceil+1$ & Corollary \ref{cor:LB-odd} \\
$2\ell+1\geq 9$ & $2\Delta+\lceil \frac{\Delta}{\ell} \rceil+5$ & \ref{t:longodd} & & \\ %\hline
 \hline
\end{tabular}
 \caption{Upper and lower bounds on $\sci(T \Box C_{k})$ for all trees $T$, where $\Delta:=\Delta(T)$.
   \label{tab:cycle}}
\end{center}
  \end{table}

Theorem \ref{t:shortodd} and Corollary \ref{cor:LB-odd} gives the upper and lower bound for the strong chromatic index of any tree with product with $C_3$. The strong chromatic index of $K_2 \square C_3$ and $P_3 \square C_3$ are $9$ and $9$ respectively. Therefore we can not improve the bounds for the strong chromatic index of product of a tree with $C_3$. For other odd cycles this bounds might be improved. We have already discussed the tightness of bounds for the strong chromatic index of product of trees with even cycles.

{\bf Acknowledgment} The fourth author thanks for the support given in part by the Sz\'echenyi 2020 programme under the project No.\ EFOP-3.6.1-16-2016-00015, and by the National Research, Development and Innovation Office -- NKFIH under the grant SNN 129364.

\bibliographystyle{plain}

\end{document}